\documentclass[14pt]{article}
\usepackage[english]{babel}
\usepackage{amsfonts,amssymb,amsmath,amstext,amsbsy,amsopn,amscd,amsthm,graphicx,euscript,graphicx}
\usepackage{graphics}
\newtheorem{thm}{Theorem}
\newtheorem{lm}{Lemma}

\newtheorem{dfn}{Definition}
\newtheorem{rk}{Remark}

\def\R{{\mathbb R}}
\def\Z{{\mathbb Z}}
\def\0{{\mathbf 0}}
\def\1{{\mathbf 1}}

\def\R{{\mathbb R}}

\title{The groups $G_{k+1}^{k}$ and fundamental groups of configuration spaces}

\author{V.O.Manturov \footnote{Research is carried out with the support of Russian Science Foundation (project no. 16-11-10291).}}

\date{}

 \def\R{{\mathbb R}}
 \def\Z{{\mathbb Z}}

\begin{document}

\maketitle

\begin{flushright}
To Leonid Arkadievich Bokut' on the occasion of his 80th birthday
\end{flushright}

AMS MSC 57M25,57M27


\begin{abstract}
In  \cite{HigherGnk}, the author has constructed natural maps from fundamental groups of  topological spaces (restricted configuration spaces) to the groups $G_{n}^{k}$. In the present paper, we show that in the case of $n=k+1$,
the group  $G_{k+1}^{k}$ is isomorphic to the fundamental group of some (quotient space of) some configuration space.
In particular, this leads to the solution of word and conjugacy problems in $G_{4}^{3}$ and sheds light on
$G_{k+1}^{k}$ for higher $k$.
\end{abstract}

\section{Introduction}

In \cite{Great}, the author defined a family of groups $G_{n}^{k}$
depending on two natural numbers $n>k$,
and formulated the following principle:
{\em if a dynamical system describing a motion of $n$ particles, admits some ``good'' codimension one property governed by exactly $k$ particles, then this dynamical system has a topological invariant valued in $G_{n}^{k}$.}
These groups are related to many problems in topology and combinatorial group theory, see, e.g., \cite{Coxeter},\cite{MN},
\cite{Imaginaire}.

For $n\in N$, let $[n]=\{1,\cdots, n\}$.

The groups  $G_{n}^{k}$ are defined as follows.
$$G_{n}^{k}=\langle a_{m}|(1),(2),(3)\rangle,$$ where
the generators $a_{m}$ are indexed by all $k$-element subsets of $[n]$,
the relation (1) means $$(a_{m})^{2}=1\;\; \mbox {for any unordered sets } m\subset \{1,\cdots,n\}, Card(m)=k;\;\;\;\eqno{(1)}$$ (2) means $$a_{m}a_{m'}=a_{m'}a_{m},\mbox{ if }Card(m\cap m')<k-1; \;\;\; \eqno{(2)}$$
and, finally, the relations (3) look as follows. For every set $U\subset [n]$ of cardinality $(k+1)$, let us order all its
 $k$-element subsets arbitrarily and denote them by $m^{1},\cdots, m^{k+1}$.
 Then (3) looks as
$$(a_{m^1}\cdots a_{m^{k+1}})^{2}=1.\;\;\; \eqno{(3)}$$
In view of (1), we can rewrite (3) in the form $$a_{m^{1}}\cdots a_{m^{k+1}}=a_{m^{k+1}}\cdots a_{m^{1}}. \eqno{(3')}$$

Note that in the case $n=k+1$, the relations (2) are void, since any two distinct subsets  $m,m'$ of $[n]$ of cardinality $k$ have intersection of cardinality  $k-1$.

In  \cite{HigherGnk}, groups  $G_{n}^{k}$ are related to fundamental groups of the following (reduced) configuration spaces. We take $n$-point sets in the $(k-1)$-dimensional real projective plane $\R{}P^{k-1}$, such that no $(k-1)$ of these points belong to the same (projective) $(k-3)$-plane.
In particular, for $k=3$ we deal  with  $n$-strand $\R{}P^{2}$-braids. There is a map from the fundamental group of the (ordered) configuration space to $G_{n}^{k}$.

These maps correspond to the good property ``some $k$ points are not in general position'' (i.e. belong the same $(k-2)$-projective plane). Note that the we have imposed the restriction that any $(k-1)$ points should be in general position.

The maps constructed in \cite{HigherGnk} have some obvious kernel corresponding to rotations of $\R{}P^{k-1}$.
It turns out that in the lowest level (for $n=k+1$), after factorisation by this kernel, we get an isomorphism of groups.

The kernel looks as follows: identifying $\R{}P^{k-1}$ with the quotient space of $S^{k-1}$ by the involution, we can just rotate $S^{k-1}$ about some axis
by $\pi$;
the corresponding motion of points is homotopically non-trivial; on the other hand, if points are in general position
from the very beginning, they remain in the general position during the rotation. Hence, no singular moment occurs and
the corresponding word is empty.

To remedy this,
we define the configuration space ${\tilde C}'_{n}(\R{}P^{k-1})$ of all unordered $n$-tuples of
points where the  first $(k-1)$ points are fixed and any $k-1$ points are in general position.
For example, for $\R{}P^{2}$, we can consider the path when one point $x_{1}$ is fixed an three other points
lie in the small neighbourhood of $x_{1}$; when rotating them around $x_{1}$, we get a non-trivial element
of the fundamental group of the corresponding configuration space, and during the whole path, there are no
moments when any three points are collinear.
Then every closed path in the configuration space giving rise to the empty word is homotopic to the trivial
path. See Lemma
\ref{lmA}.


More precisely, we shall prove the following
\begin{thm}
There is a subgroup ${\tilde G}_{k+1}^{k}$
of the group $G_{k+1}^{k}$ of index $2^{k-1}$
is isomorphic to $\pi_{1}({\tilde C}'_{k+1}(\R{}P^{k-1}))$.\label{th0}
\end{thm}

The space ${\tilde C}'_{n}$ and the group ${\tilde G}_{k+1}^{k}$ will be defined later.

The simplest case of the above Theorem is
\begin{thm}
The group ${\tilde G}_{4}^{3}$
(subgroup of index four of $G_{4}^{3}$)  is isomorphic to $\pi_{1}(FBr_{4}(\R P^{2}))$,
the $4$-strand pure braid group on $\R{}P^{2}$ with two point fixed.\label{th1}
\end{thm}

The paper is organized as follows. In the next section, we give all necessary definitions and
construct maps from configuration spaces of points in $\R{}P^{k-1}$ to $G_{k+1}^{k}$. In Section 3, we prove Theorem \ref{th0}.
In Section 4, we shall discuss an algebraic lemma about reduction of the word problem in $G_{k+1}^{k}$
to the word problem in some subgroup of it.


We conclude the paper by Section 4 by discussing some open problems for further research.

\section{Basic definitions}

Let us now pass to the definition of spaces $C'_{n}(\R{}P^{k-1})$
and maps from the corresponding fundamental groups to
the groups $G_{n}^{k}$.

Let us fix a pair of natural number $n>k$. A point in $C'_{n}(\R{}P^{k-1})$
is an ordered set of $n$ pairwise distinct points in  $\R{}P^{k-1}$, such that any $(k-1)$
of them are in general position. Thus, for instance, if $k=3$,
then the only condition is that these points are pairwise distinct. For $k=4$
 for points $x_{1},\dots, x_{n}$ в $\R{}P^{3}$
 we impose the condition that no three of them belong
 to the same line (though some four are allowed to belong to the same plane),
 and for  $k=5$ a point in $C'_{n}(\R{}P^{4})$ is a set of ordered $n$ points in  $\R{}P^{4}$,
 with no four of them belonging to the same $2$-plane.

Let us use the projective coordinates $(a_{1}:a_{2}:\cdots: a_{k})$ in $\R{}P^{k-1}$
and let us fix the following  $k-1$ points in general position, $y_{1},y_{2}, \cdots, y_{k-1}\in \R{}P^{k-1}$,
where $a_{i}(y_{j})=\delta_{i}^{j}$.
Let us  define
the subspace ${\tilde C}'_{n}(\R{}P^{k-1})$ taking those $n$-tuples of points
$x_{1},\cdots, x_{n}\in \R{}P^{k-1}$ for which $x_{i}=y_{i}$ for $i=1,\cdots, k-1$.

 We say that a point $x\in C'_{n}(\R^{k-1})$  is {\em singular}, if the set of points $x=(x_{1},\dots, x_{n})$, corresponding $x$, contains some subset of $k$ points lying on the same $(k-2)$-plane. Let us fix two  non-singular points $x,x'\in C'_{n}(\R{}P^{k-1}).$

We shall consider smooth paths $\gamma_{x,x'}: [0,1]\to C'_{n}(\R^{k-1})$.
For each such path there are values $t$ for which $\gamma_{x,x'}(t)$ is not in the general position (some $k$ of them belong to the same $(k-2)$-plane). We call these values $t\in [0,1]$ {\em singular}.

On the path $\gamma$, we say that the moment $t$ of passing through the singular point $x$, corresponding to the set  $x_{i_1},\cdots, x_{i_k}$, is {\em transverse} (or stable) if for any sufficiently small perturbation ${\tilde \gamma}$ of the path  $\gamma$, in the neighbourhood of the moment $t$ there exists exactly one time moment $t'$ corresponding to some set of points $x_{i_1},\cdots, x_{i_k}$ not in general position.

\begin{dfn}
We say that a path is {\em good and stable} if the following holds:

\begin{enumerate}

\item
The set of singular values $t$ is finite;

\item For every singular value  $t=t_{l}$ corresponding to $n$ points representing $\gamma_{x,x'}(t_{l})$, there exists only one subset of $k$ points belonging to a $(k-2)$-plane;

\item Each singular value is  {\em stable}.

\end{enumerate}
\end{dfn}

\begin{dfn}
We say that the path without singular values is {\em void}.
\end{dfn}

\begin{dfn}
By a {\em braid} we mean a smooth path in the above configuration space whose
initial and terminal configurations of points coincide as sets.
We say that a braid whose ends $x,x'$ coincide with respect to the order, is {\em pure}. We say that two  braids $\gamma,\gamma'$ with endpoints $x,x'$ are {\em isotopic}
 if there exists a continuous family  $\gamma^{s}_{x,x'},s \in [0,1]$
 of smooth paths with fixed ends such that  $\gamma^{0}_{x,x'}=\gamma,\gamma^{1}_{x,x'}=\gamma'$.
\end{dfn}

 \begin{rk}
 By a small perturbation, any path can be made good and stable (if endpoints are generic,
 we may also require that the endpoints remain fixed).
 \end{rk}

\begin{dfn}
A path from $x$ to $x'$ is called {\em a braid} if the points representing $x$ are
the same as those representing $x'$ (possibly, in different orders); if $x$ coincides
with $x'$ with respect to order, then such a braid is called {\em pure}.
\end{dfn}

 There is an obvious concatenation structure on the set of braids: for paths $\gamma_{x,x'}$ and $\gamma'_{x',x''}$,
  the concatenation is defined as a path $\gamma''_{x,x''}$ such that $\gamma''(t)=\gamma(2t)$ for $t\in [0,\frac{1}{2}]$ and $\gamma''(t)=\gamma'(2t-1)$ for $t\in [\frac{1}{2},1]$; this path can be smoothed in the neighbourhood of  $t=\frac{1}{2}$; the result of such smoothing is well defined up to isotopy.

  Thus, the sets of braids and pure braids (for fixed $x$) admit
  a group structure.  This latter
  group is certainly isomorphic to the fundamental group  $\pi_{1}(C'_{n}(\R^{k-1}))$.
  The former group is isomorphic to the fundamental group of the quotient space by the action of the permutation group.

\section{The realisability of  $G_{k+1}^{k}$}

The main idea of the proof of Theorem \ref{th0} is to associate with every word in $G_{k+1}^{k}$
a braid in ${\tilde C}'_{k+1}(\R{}P^{k-1})$.

Let us start with the main construction from \cite{HigherGnk}.

With each good and stable path from $PB_{n}(\R{}P^{2})$
we associate an element of the group  $G_{n}^{k}$ as follows. We enumerate all singular values of
 our path $0<t_{1}<\dots <t_{l}<1$ (we assume than  $0$ and $1$
 are not singular). For each singular value $t_{p}$ we have a set $m_{p}$ of $k$ indices corresponding
 to the numbers of points which are not in general position. With this value we associate the letter  $a_{m_{p}}$. With the whole
 path  $\gamma$ (braid) we associate the product $f(\gamma)=a_{m_{1}}\cdots a_{m_{l}}$.

\begin{thm}\cite{HigherGnk}
The map   $f$ takes isotopic braids to equal elements
of the group  $G_{n}^{k}$.
For pure braids, the map   $f$ is a homomorphism  $f:\pi_{1}(C'_{n}(\R{}P^{2}))\to G_{n}^{k}$.
\label{thgn3}
\end{thm}
We are interested in the case $k=3$.

Now we claim that

{\em Every word from  $G_{k+1}^{k}$ can be realised by a path of the above form.}

Note that if we replace $\R{}P^{k-1}$ with $\R^{k-1}$, the corresponding statement will fail.
Namely, starting with the configuration of four points, $x_{i},i=1,\cdots, 4$, where
$x_{1},x_{2},x_{3}$ form a triangle and $x_{4}$ lies inside this triangle, we see that
any path starting from this configuration will lead to a word starting from $a_{124},a_{134},
a_{234}$ but not from $a_{123}$. In some sense the point $4$ is ``locked'' and the points
are not in the same position.

\begin{figure}
\centering\includegraphics[width=100pt]{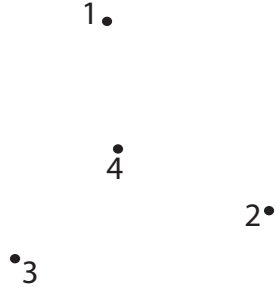}
\caption{The ``locked'' position for the move $a_{123}$}
\label{locked}
\end{figure}

The following well known theorem (see, e.g., \cite{Wu}) plays a crucial role in the construction
\begin{thm}
For any two sets of $k+1$ points in general position in $\R{}P^{k-1}$,
$(x_{1},\cdots, x_{k+1})$ and $(y_{1},\cdots, y_{k+1})$ there is an action of
$PGL(k, \R)$ taking all $x_{i}$ to $y_{i}$.\label{wulem}
\end{thm}

For us, this  will mean that there is no difference between all possible ``non-degenerate
starting positions'' for $k+1$ points in $\R{}P^{k}$.

We shall deal with paths in ${\tilde C'}_{k+1}(\R{}P^{k-1})$ similar to braids. Namely,
we shall fix a set of $2^{k-1}$ points such that all paths will start and end at these points.

We shall denote homogeneous coordinates in $\R{}P^{k-1}$ by $(a_{1}:\cdots: a_{k})$ in
contrast to points (which we denote by $(x_{1},\cdots, x_{k+1})$).
\subsection{Constructing a braid from a word in $G_{k+1}^{k}$}

Our main goal is to construct a braid by a word. To this end, we need a base point for the braid.
For the sake of convenience, we shall use not one, but rather $2^{k-1}$ reference points.
For the first $k$ points $y_{1}=(1:0:\cdots:0),\cdots, y_{k}=(0:\cdots:0:1)$ fixed, we
will have $2^{k-1}$ possibilities for the choice of the last point. Namely, let us
consider all possible strings  of length $k$ of $\pm 1$ with the last coordinate $+1$:
$(1,1,\cdots, 1,1),(1,\cdots, 1,-1,1),\cdots, (-1,-1,\cdots, -1,1)$ with $a_{k}\sim 1$. We shall denote
these points by $y_{s}$ where $s$ records the first $(k-1)$ coordinates of
the point.

Now, for each string $s$ of length $k$ of $\pm 1$, we set $z_{s}=(y_{1},y_{2},\cdots, y_{k},y_{s})$.

The following lemma is evident.
\begin{lm}
For every point $z\in \R{}P^{k-1}$ with projective coordinates
$(a_{1}(z):\cdots: a_{k-1}(z):1)$, let ${\tilde z}=(sign(a_{1}(z)):sign(a_{2}(z)):
\cdots: sign(a_{k-1}(z)):1)$. Then there is a path
between $(y_{1},\cdots, y_{k},z)$ and $(y_{1},\cdots, y_{k}, {\tilde z})$
in ${\tilde C}'_{k+1}(\R{}P^{k})$ with the first points $y_{1},\cdots, y_{k}$
fixed, and the corresponding path in ${\tilde C}'_{k+1}$ is void.
\label{lemmaB}
\end{lm}

\begin{proof}
Indeed, it suffices just to connect $z$ to ${\tilde z}$ by a geodesic.
\end{proof}

From this we easily deduce the following
\begin{lm}
Every  point $y\in {\tilde C}'_{k+1}(\R{}P^{k-1})$ can be connected by a void
path to some
$(y_{1},\cdots, y_{k}, y_{s})$ for some $s$.\label{lemBB}
\end{lm}

\begin{proof}
Indeed, the void path can be constructed in two steps. At the first step, we
construct a path which moves both $y_{k}$ and $y_{k+1}$, so that $y_{k}$
becomes $(0:\cdots 0:1)$, and at the second step, we use Lemma \ref{lemmaB}.
To realise the first step, we just use linear maps which keep the hyperplane
$a_{k}=0$ fixed.
\end{proof}

The lemma below shows that the path mentioned in Lemma \ref{lemBB} is unique up to homotopy.

\begin{lm}
Let $\gamma$ be a closed path in ${\tilde C}'_{k+1}(\R{}P^{k-1})$ such that
the word $f(\gamma)$ is empty. Then $\gamma$ is homotopic to the trivial braid.
\label{lmA}
\end{lm}

\begin{proof}

In ${\tilde C}'_{k+1}$, we deal with the motion of points, where all but $x_{k},x_{k+1}$ are fixed.

Consider the projective hyperplane ${\cal P}_{1}$ passing through $x_{1},\cdots, x_{k-1}$ given
by the equation $a_{k}=0$.

We know that none of the points $x_{k}, x_{k+1}$ is allowed to belong to ${\cal P}_{1}$.
Hence, we may fix the last coordinate $a_{k}(x_{k})=a_{k}(x_{k+1})=1$.

Now, we may pass to the affine coordinates of these two points (still to be denoted
by $a_{1},\cdots, a_{k}$).

Now, the condition $\forall i=1,\cdots, k-1: a_{j}(x_{k})\neq a_{j}(x_{k+1})$ follows
from the fact that the points $x_{1},\cdots, {\hat{x_{j} }}, x_{k+1}$ are generic.

This means that $\forall i=1,\cdots, k-1$ the sign of  $a_{i}(x_{k})-a_{i}(x_{k+1})$ remains fixed.

Now, the motion of points $x_{k},x_{k+1}$ is determined by their coordinates
$a_{1},\cdots, a_{j}$, and since their signs are fixed, the configuration space
for this motions is simply connected.

This means that the configuration space for  $\gamma$ 
is the direct product of the two configuration spaces: the one for the point $x_{k}$ and
the one for the point $x_{k+1}$. Having $x_{1},\cdots, x_{k-1}$ fixed on $\R{}P{k-1}$,
the configuration space of admissible positions for $x_{k}$ is nothing but  
the hemisphere; for fixed $x_{k}$, the admissible positions for $x_{k+1}$ are regulated
by inequalities in coordinates; this space is also simply connected.

\end{proof}

\subsection{The surjectivity}

In order to prove that some map is a bijection, it is crucial to understand the surjectivity.
For the case of braid groups and words in $G_{n}^{k}$, this surjectivity is attained in the case of $n=k+1$.

Namely, our next strategy is as follows. Having a word in $G_{k+1}^{k}$, we shall associate
with this word a path in ${\tilde C}'_{k+1}(\R{}P^{k-1})$. After each letter,
we shall get to  $(y_{1},\cdots ,y_{k},y_{s})$ for some $s$.

Let us start from $(y_{1},\cdots, y_{k},y_{1,\cdots, 1})$.


After making the final step, one can calculate the coordinate of the $(k+1)$-th
points. They will be governed by Lemma
\ref{governed} (see ahead). As we shall see later, those words we have to check for
the solution of the word problem in $G_{k+1}^{k}$, will lead us to closed paths,
i.e., pure braids.

Let us be more detailed.


\begin{lm}
Given a non-singular set of points $y$ in $\R{}P^{k-1}$.
Then for every set of $k$ numbers $i_{1},i_{2},\cdots, i_{k}\in [n]$,
 there exists a path $y_{i_{1}\cdots i_{k}}(t)= y(t)$ in $C'_{n}(\R{}P^{k-1})$, having   $y(0)=y(1)=y$
 as the starting point and the final point and
 with only one singular moment corresponding to the numbers $i_{1},\cdots, i_{k}$ which
 encode the points not in general position; moreover, we may assume that at this moment
  all points except $i_{1}$,
  are fixed during the path.

 Moreover, the set of paths possessing this property is connected:
   any other path ${\tilde y}(t)$,
 possessing the above properties, is homotopic to $y(t)$ in this class.
 \label{lm2}
\end{lm}

\begin{proof}
Indeed, for the first statement of the Lemma,
it suffices to construct a path for some initial position of points and
then apply Theorem \ref{wulem}.

For the second statement, let us take two different paths $\gamma_{1}$ and
$\gamma_{2}$ satisfying the conditions of the Lemma. By a small perturbation, we may assume that for both of them,
 $t=\frac{1}{2}$ is a singular moment with the same position of $y_{i_{1}}$.

 Now, we can contract the loop formed by $\gamma_{1}|_{t\in [\frac{1}{2},1]}$
and the inverse of $\gamma_{2}|_{t\in [\frac{1}{2},1]}$ by using Lemma \ref{lmA} as
this is a small perturbation of a void braid. We are left with
$\gamma_{1}|_{t\in [0,\frac{1}{2}]}$ and the inverse of $\gamma_{2}|_{t\in [0,\frac{1}{2}]}$
which is contractible by Lemma \ref{lmA} again.

\end{proof}

\begin{rk}
Note that in the above lemma, we deal with the space $C'_{n}(\R{}P^{k-1})$,
not with ${\tilde C}'_{n}(\R{}P^{k-1})$. On the other hand, we may always choose
$i_{1}\in \{k,k+1\}$; hence, the path in question can be chosen in ${\tilde C}'(\R{}P^{k-1})$.
\end{rk}

Now, for every subset
 $m\subset [n], Card(m)=k+1$ we can create a path
 $p_{m}$ starting from any of the base points listed
 above and ending at the corresponding basepoints.

Now, we construct our path step-by step by applying Lemma \ref{lm2} and returning
to some of base points by using Lemma \ref{lemBB}.

From \cite{HigherGnk}, we can easily get the following
\begin{lm}
Let $i_{1},\cdots, i_{k+1}$ be some permutation of $1,\cdots, k+1$. Then the concatenation
of paths $p_{i_{1}i_{2}\cdots i_{k}}p_{i_{1}i_{3}i_{4}\cdots i_{k+1}}\cdots
p_{i_{2}i_{3}\cdots i_{k}}$ \\ is homotopic to the concatenation of paths in the inverse
order $p_{i_{2}i_{3}\cdots i_{k}}\cdots p_{i_{1}i_{3}i_{4}\cdots i_{k+1}}p_{i_{1}i_{2}\cdots i_{k}}.$\label{quadrisec}
\end{lm}

\begin{proof}
Indeed, in \cite{HigherGnk}, some homotopy corresponding to the above mentioned relation
corresponding to {\em some} permuation is discussed. However, since all basepoints are
similar to each other as discussed above, we can transform the homotopy from \cite{HigherGnk}
to the homotopy for any permutation.
\end{proof}

\begin{lm}
For the path starting from the point $(y_{1},\cdots, y_{k},y_{s})$ constructed as in
Lemma \ref{lm2} for the set of indices $j$,
we get to the point $(y_{1},\cdots, y_{k},y_{s'})$ such that:
\begin{enumerate}
\item if $j=1,\cdots, k$, then $s'$ differs from $s$ only in coordinate $a_{j}$;
\item if $j=k+1$, all coordinates of $s'$ differ from those coordinates of $s$ by sign.
\end{enumerate}
\label{governed}
\end{lm}

\begin{rk}
Note that in general the path constructed in the above lemma is generally not a braid, but
almost a braid: for the fixed position of the initial point, there are only infinitely many possibilities
for the position of the final point.
\end{rk}

Denote the map from words in $G_{k+1}^{k}$ to paths between basepoints by $g$.

By construction, we see that for every word $w$ we have $f(g(w))=w\in G_{k+1}^{k}$.

Now, we define the group ${\tilde G}_{k+1}^{k}$ as the subgroup of $G_{k+1}$ which is
taken by $g$ to {\em braids}, i.e., to those paths with coinciding initial and final
points. From lemma \ref{governed}, we see that this is a subgroup of index $(k-1)$: there are exactly
$(k-1)$ coordinates.

\subsection{Equal words lead to homotopic paths}

Let us pass to the proof of Theorem \ref{th1}.
Our next goal is to see that equal words can originate only from homotopic paths.

To this end, we shall first show that the map $f$
from Theorem \ref{thgn3}
is an isomorphism for $n=k+1$. To perform this goal, we should construct the
inverse map $g:{\tilde G}_{k+1}^{k}\to \pi_{1}({\tilde C'}_{k+1}(\R{}P(k-1)))$.

Note that for $k=3$ we deal with
the pure braids $
PB_{4}(\R{}P^{2}). $

Let us fix a point $x\in C'_{4}(\R{}P^{2})$.
With each generator $a_{m},m\subset [n],Card(m)=k$ we associate a path  $g(m)=y_{m}(t)$,
described in Lemma \ref{lm2}. This path is not a braid: we can return to any
of the $2^{k-1}$ base points. However, once we take the concatenation of paths correspoding
to ${\tilde G}_{k+1}^{k}$, we get a braid.

By definition of the map $f$, we have $f(g(a_{m}))=a_{m}$.
Thus, we have chosen that the map $f$ is a surjection.

Now, let us prove that the kernel of the map $f$ is trivial. Indeed,
assume there is a pure braid $\gamma$ such that $f(\gamma)=1\in G_{k+1}^{k}$.
We assume that $\gamma$ is good and stable. If this path has $l$ critical points, then
we have the word corresponding to it  $a_{m_1}\cdots a_{m_l}\in G_{k+1}^{k}$.

Let us perform the transformation $f(\gamma)\to 1$ by applying the relations of $G_{k+1}^{k}$
to it and transforming the path $\gamma$ respectively. For each relation of the sort $a_{m}a_{m}=1$ for a generator $a_{m}$ of the group
$G_{k+1}^{k}$, we see that the path $\gamma$ contains two segments whose concatenation
is homotopic to the trivial loop (as follows from the second statement of Lemma 4).

Whenever we have a relation of length $2k+2$ in the group $G_{k+1}$, we use the Lemma
\ref{quadrisec} to perform the homotopy of the loops.

 Thus, we have proved that if the word $f(\gamma)$ corresponding to a braid $\gamma\in G_{k+1}^{k}$
is equal to $1$ in $G_{k+1}^{k}$ then the braid $\gamma$ is isotopic to a braid $\gamma'$
such that the word corresponding to it is empty.
Now, by Lemma \ref{lmA}, this braid is homotopic to the trivial braid.

\section{The group $H_{k}$ and the algebraic lemma}

The aim of the present section is to reduce the word problem in $G_{k+1}^{k}$ to the word problem
in a certain subgroup of it, denoted by $H_{k}$.

Let us rename all generators of $G_{k+1}^{k}$ lexicographically:
$b_{1}=a_{1,2,\cdots,k},\cdots, b_{k+1}=a_{2,3,\cdots, k+1}$.

Let $H_{k}$ be the subgroup of $G_{k+1}^{k}$ consisting of all elements $x\in G_{k+1}^{k}$ that can be
represented by words with no occurencies of the last letter $b_{k+1}$.

Our task is to understand whether a word in $G_{k+1}^{k}$ represents an element in $H$.
To this end, we recall the map from \cite{MN}. Consider the group $F_{k}=\Z_{2}^{*2^{k}}=\langle
c_{m}|c_{m}^{2}=1\rangle$, where all generators $c_{m}$ are indexed by $k$-element strings of $0$
and $1$ with only relations being that the square of each generator is equal to $1$.
We shall construct a map\footnote{This map becomes a homomorphism when restricted to a
finite index subgroup} from $G_{k+1}^{k}$ to $F_{k}$ as follows.

Take a word $w$ in generators of $G_{k+1}^{k}$ and list all occurencies of the last letter
$b_{k+1}=a_{2,\cdots, k+1}$ in this word. With any such occurency we first associate the string of indices $0,1$
of length $k$. The $j$-th index is the number of letters $b_{j}$ preceding this occurency of $b_{k+1}$ modulo $2$.
Thus, we get a string of length $k$.

Now, we identify ``opposite pairs'' of strings: we set $(x_{1},\cdots, x_{k})\sim(x_{1}+1,\cdots, x_{k}+1)$.
Now, we may think that the last ($k$-th) element of our string is always $0$, so, we can restrict ourselves
with $(x_{1},\cdots, x_{k-1},0)$. Such a string of length $k-1$ is called the {\em index} of
the occurency of $b_{k+1}$.

Having this done, we associate with each occurency of $b_{k+1}$ having index $m$ the generator $c_{m}$ of $F_{k}$.
With the word $w$, we associate the word $f(w)$ equal to the product of all generators $c_{m}$ in order.

In \cite{MN}, the following Lemma is proved
\begin{lm}
The map $f:G_{k+1}^{k}\to F_{k}$ is well defined.
\end{lm}

Now, let us prove the following
\begin{lm}
If $f(w)=1$ then $w\in H_{k}$.
\end{lm}

In other words, the free group $F_{k}$ yields the only obstruction for an element from $G_{k+1}^{k}$ to have a presentation
with no occurency of the last letter.

\begin{proof}
Let $w$ be a word such that $f(w)=1$. If $f(w)$ is empty, then there is nothing to prove. Since $f$ is trivial in $F_{k}$, it contains some two adjacent occurencies of
the same generator $b_{x},\cdots, b_{x}$. This means that in the initial word $w$ is of the type
$w=A b_{k+1} B b_{k+1} C$, where $A$ and $C$ are some words, and $B$ contains no occurencies of $b_{k+1}$
and the number of occurencies of $b_{1},b_{2},\cdots, b_{k}$ in $B$ are of the same parity.

Our aim is to show that $w$ is equal to a word with smaller number of $b_{k+1}$ in $G_{k+1}^{k}$. Then we will be able to induct
on the number of $b_{k+1}$ until we get a word without $b_{k+1}$.

Thus, it suffices for us to show that
$b_{k+1}Bb_{k+1}$ is equal to a word from $H_{k}$. We shall induct on the length of $B$.
Without loss of generality, we may assume that $B$ is reduced, i.e., it does not contain
adjacent $b_{j}b_{j}$.

Let us read the word $B$ from the very beginning $B=b_{i_1}b_{i_2}\cdots$ If all elements
$i_{1},i_{2},\cdots$ are distinct, then, having in mind that the number of occurencies of all generators
in $B$ should be of the same parity, we conclude that $b_{k+1}B= B^{-1} b_{k+1}$, hence $b_{k+1}Bb_{k+1}=B^{-1}b_{k+1}b_{k+1}=
B^{-1}$
is a word without occurencies of $b_{k+1}$.

Now assume ${i_1}={i_p}$. Then we have
$b_{k+1}B=b_{k+1}b_{i_1}\cdots b_{i_{p-1}} b_{i_1} B'$. Now we collect all indices distinct from
$i_{1},\cdots, i_{p-1},{k+1}$ and write the word $P$ containing exactly one generator for each index
(the order does not matter). Then the word $W = P b_{k+1} b_{i_{1}}\cdots b_{i_{p-1}}$ contains
any letter exactly once and we can invert the word $W$ as follows: $W^{-1}=b_{i_{p-1}}\cdots b_{i_{1}}b-{k+1}P$.
Thus, $b_{k+1}B=P^{-1}(Pb_{k+1}b_{i_1}\cdots b_{i_{p-1}})b_{i_1}B'=P^{-1}b_{i_{p-1}}\cdots b_{i_1}b_{k+1}P^{-1}b_{i_1}B'$.

We know that the letters in $P$ (hence, those in $P^{-1}$) do not contain $b_{i_1}$. Thus,
the word $P^{-1} b_{i_1}$ consists of different letters. Now we perform the same trick:
create the word $Q=b_{i_2}b_{i_3}\cdots b_{i_{p-1}}$ consisting of remaining letters from $\{1,\cdots, k\}$,
we get:

$$ P^{-1}b_{i_{p-1}}\cdots b_{i_1}b_{k+1}P^{-1}b_{i_1}B'$$

$$=P^{-1}b_{i_{p-1}}\cdots b_{i_1}Q^{-1}Qb_{k+1}P^{-1}b_{i_1}B'$$

$$=P^{-1}b_{i_{p-1}}\cdots b_{i_1}Q^{-1}b_{i_1}Pb_{k+1}QB'.$$

Thus, we have moved $b_{k+1}$ to the right and between the two occurencies
of the letter
$b_{k+1}$, we replaced $b_{i_1}\cdots, b_{i_{p-1}}b_{i_1}$ with
just $b_{i_2}\cdots b_{i_{p-1}}$, thus, we have shortened the distance
between the two adjacent occurencies of $b_{k+1}$.

Arguing as above, we will finally cancel these two letters and perform the
induction step.

\end{proof}

\begin{thm}
If the word problem in $H_{k}$ is solvable, then the word problem in $G_{k+1}^{k}$ is solvable.
Moreover, the solution of the latter is constructive once the solution of the former is constructive.
\end{thm}

\begin{proof}
Indeed, having a word $w$ in generators of $G_{k+1}^{k}$, we can look at the image of this
word by the map $f$ in the group $H_{k}$. If it is not equal to $1$, then, from \cite{MN},
it follows that $w$ is non-trivial, otherwise we can construct a word ${\tilde w}$ in $H_{k}$
equal to $w$ in $G_{k+1}^{k}$.
\end{proof}

Certainly, to be able to solve the word problem in $H_{k}$, one needs to know a presentation for
$H_{k}$. It is natural to take $b_{1},\cdots, b_{k}$ for generators of $H_{k}$. Obviously,
they satisfy the relations $b_{j}^{2}=1$ for every $j$.

To understand the remaining relations for different $k$, we shall need geometrical arguments.

\section{Concluding remarks}

We have completely constructed the isomorphism between the (finite index subgroup) of
the group $G_{k+1}^{k}$ and a fundamental group of some configuration space.

This completely solves the word problem in $G_{4}^{3}$ for braid groups in projective spaces are very
well studied, see, e.g., \cite{Bu,GG1,GG2,GG3}. The same can be said about the conjugacy problem in ${\tilde G}_{4}^{3}$.

Besides, we have seen that the word problem for the case of general $G_{k+1}^{k}$
can be reduced to the case of $H_{k}$.

In a subsequent paper, we shall completely describe the relations in $H_{k}$ by geometric
reasons and apply it to the word problem in $G_{5}^{4}$.
Here we just mention that it was proved by A.B.Karpov (unpubished)
that the only relations in $H_{3}$ are $a^{2}=b^{2}=c^{2}$ which also follows
from the geometrical techniques of the present paper.

The main open question which remains unsolved is how to construct a configuration
space which can realise $G_{n}^{k}$ for $n>k+1$. Even in the case of $G_{4}^{2}$
this problem seems very attractive though the word and conjugacy  problems for $G_{n}^{2}$ can
be solved by algebraic methods, see \cite{Coxeter}.

The word problem and the partial case of conjugacy problem in $G_{4}^{3}$ was first
solved by A.B.Karpov, but the author has not yet seen any complete text of
it.

It would be very interesting to compare the approach of $G_{n}^{k}$ with various generalizations
of braid groups, e.g., Manin-Schechtmann groups \cite{ManinSchechtmann,KapranovVoevodsky}.

The author is very grateful to I.M.Nikonov, L.A.Bokut' and Jie Wu for extremely useful discussions
and comments.

 \end{document}